\documentclass[12pt,reqno,lineno,oneside]{amsart}
\usepackage{amsmath, amssymb}
\usepackage{cite}
\usepackage{hyperref}

\newtheorem{theorem}{Theorem}[section]
\newtheorem{corollary}{Corollary}[section]
\newtheorem{lemma}{Lemma}[section]
\theoremstyle{definition}
\newtheorem{definition}{Definition}[section]
\theoremstyle{remark}
\newtheorem{remark}{Remark}[section]
\numberwithin{equation}{section}
\newtheorem{case}{Case}[section]

\input{tcilatex}
\thanks{$^*$ Corresponding author : N. Magesh (nmagi\_2000@yahoo.co.in).
\ \newline
The authors would like to thank Prof. H. Orhan, Department of Mathematics, Faculty of Science, Ataturk University, 25240 Erzurum, Turkey for his guidance and support.
}

\begin{document}
\title[Hankel Determinant functionals for bi-univalent functions]{Fekete-Szeg%
\"{o} problem and Second Hankel Determinant for a class of bi-univalent functions}
\author{N. Magesh and J. Yamini}
\address{Post-Graduate and Research Department of Mathematics
\newline Government Arts College for Men
\newline Krishnagiri 635001, Tamilnadu, India.
\newline E-Mail address: $nmagi\_2000@yahoo.co.in$
}
\address{Department of Mathematics,\newline Government First Grade College
\newline Vijayanagar, Bangalore-560104, Karnataka, India.
\newline E-Mail address: $yaminibalaji@gmail.com$
}
\maketitle

\begin{abstract}
In this sequel to the recent work (see Azizi et al., 2015), we investigate
a subclass of analytic and bi-univalent functions in the open unit disk.
We obtain bounds for initial coefficients, the Fekete-Szeg\"{o} inequality and
the second Hankel determinant inequality for functions belonging to this
subclass. We also discuss some new and known special cases, which can be
deduced from our results.
\end{abstract}


\section{Introduction}

Let $\mathcal{A}$ denote the class of functions of the form
\begin{equation}
f(z)=z+\sum\limits_{n=2}^{\infty }a_{n}z^{n}  \label{Int-e1}
\end{equation}%
which are analytic in the open unit disc $\mathbb{U}=\{z:z\in \mathbb{C}\,\,%
\mathrm{and}\,\,|z|<1\}$ and let $\mathcal{S}$ denote the class of functions
in $\mathcal{A}$ that are univalent in $\mathbb{U}.$

For two functions $f$ and $g,$ analytic in $\mathbb{U},$ we say that the
function $f$ is subordinate to $g$ in $\mathbb{U},$ and write $f\prec g,$ if
there exists a Schwarz function $w,$ analytic in $\mathbb{U},$ with $w(0)=0$
and $|w(z)|<1$ such that $f(z)=g(w(z))$; $z,w\in \mathbb{U}$. In particular,
if the function $g$ is univalent in $\mathbb{U},$ the above subordination is
equivalent to $f(0)=g(0)$ and $f(\mathbb{U})\subset g(\mathbb{U}).$


Let $\varphi $ be an analytic and univalent function with positive real part
in $\mathbb{U}$, $\varphi (0)=1,$ $\varphi ^{\prime }(0)>0$ and $\varphi $
maps the unit disk $\mathbb{U}$ onto a region starlike with respect to $1$
and symmetric with respect to the real axis. The Taylor's series expansion
of such function is
\begin{equation}
\varphi (z)=1+B_{1}z+B_{2}z^{2}+B_{3}z^{3}+\dots ,  \label{varphi-expression}
\end{equation}%
where all coefficients are real and $B_{1}>0.$ Throughout this paper we
assume that the function $\varphi $ satisfies the above conditions unless
otherwise stated.

By $\mathcal{S}^{\ast }(\varphi )$ and $\mathcal{K}(\varphi )$ we denote the
following classes:
\begin{equation*}
\mathcal{S}^{\ast}(\varphi ):=\left\{ f\in \mathcal{S}:\,\,\frac{zf^{\prime
}(z)}{f(z)}\prec \varphi (z);\,\,z\in \mathbb{U}\right\}
\end{equation*}%
and
\begin{equation*}
\mathcal{K}(\varphi ):=\left\{ f\in \mathcal{S}:1+\frac{zf^{\prime \prime
}(z)}{f^{\prime }(z)}\prec \varphi (z);\,\,z\in \mathbb{U}\right\} .
\end{equation*}

The classes $\mathcal{S}^{\ast }(\varphi )$ and $\mathcal{K}(\varphi )$ are
the extensions of a classical set of starlike and convex functions (e.g. see
Ma and Minda \cite{Ma-Minda}). For $0 \leq \beta < 1,$ the classes $\mathcal{%
S}^{\ast}(\beta):= \mathcal{S}^{\ast}\left(\frac{1+(1-2\beta)z}{1+z}\right)$
and $\mathcal{K}(\beta):= \mathcal{K}\left(\frac{1+(1-2\beta)z}{1+z}\right)$
are starlike and convex functions of order $\beta$.

It is well known (e.g. see Duren \cite{Duren}) that every function $f\in
\mathcal{S}$ has an inverse map $f^{-1},$ defined by
$f^{-1}(f(z))=z,$ $z\in \mathbb{U}$ and
$f(f^{-1}(w))=w,$ $(|w|<r_{0}(f);\,\,r_{0}(f)\geqq \frac{1}{4}),$
where
\begin{equation}
f^{-1}(w)=w-a_{2}w^{2}+(2a_{2}^{2}-a_{3})w^{3}-(5a_{2}^{3}-5a_{2}a_{3}+a_{4})w^{4}+\ldots .
\label{Int-f-inver}
\end{equation}

A function $f\in \mathcal{A}$ is said to be bi-univalent in $\mathbb{U}$ if
both $f$ and $f^{-1}$ are univalent in $\mathbb{U}.$ We let $\sigma $ denote
the class of bi-univalent functions in $\mathbb{U}$ given by (\ref{Int-e1}).
A function $f$ is said to be bi-starlike of
Ma-Minda type or bi-convex of Ma-Minda type if both $f$ and $f^{-1}$ are,
respectively, of Ma-Minda starlike or convex type. These classes are
denoted, respectively, by $\mathcal{S}_{\sigma }^{\ast }(\varphi )$ and $%
\mathcal{K}_{\sigma }(\varphi )$ (see \cite{Ali-Ravi-Ma-Mina-class}).
For $0\leqq \beta <1,$ a function $f\in \sigma $ is in the class $S_{\sigma
}^{\ast }(\beta)$ of bi-starlike functions of order $\beta,$ or $%
\mathcal{K}_{\sigma}(\beta)$ of bi-convex functions of order $\beta$ if
both $f$ and its inverse map $f^{-1}$ are, respectively, starlike or convex
of order $\beta.$ For a history and examples of functions which are (or which are not) in the
class $\sigma $, together with various other properties of subclasses of
bi-univalent functions one can refer \cite{Ali-Ravi-Ma-Mina-class,Bulut,Caglar-Orhan,BAF-MKA,%
HO-NM-VKB,Peng,HMS-AKM-PG,Zaprawa}.%


For integers $n \geq 1$ and $q \geq 1,$ the $q-$th Hankel determinant,
defined as
\begin{equation*}  \label{HD}
H_{q}(n)=\left|%
\begin{array}{cccc}
a_{n} & a_{n+1} & \cdots & a_{n+q-1} \\
a_{n+1} & a_{n+2} & \cdots & a_{n+q-2} \\
\vdots & \vdots & \vdots & \vdots \\
a_{n+q-1} & a_{n+q-2} & \cdots & a_{n+2q-2}%
\end{array}
\right| \qquad (a_1=1).
\end{equation*}
The Hankel determinant plays an important role in the study of singularities
(see \cite{Dienes-1957}). This is also an important in the study of power
series with integral coefficients \cite{Cantor,Dienes-1957}.
The Hankel determinants $H_2(1) = a_3 - a_2^2$ and $H_2(2) = a_2a_4 - a_2^3$
are well-known as Fekete-Szeg\"{o} and second Hankel determinant functionals
respectively. Further Fekete and Szeg\"{o} \cite{Fekete-Szego} introduced
the generalized functional $a_3-\delta a_2^2,$ where $\delta$ is some real
number. In 1969, Keogh and Merkes \cite{Keogh-Merkes} discussed the
Fekete-Szeg\"{o} problem for the classes starlike and convex functions.
Recently, several authors have investigated upper bounds for the Hankel
determinant of functions belonging to various subclasses of univalent
functions \cite{SHD-Ali-2009,VKD-RRT-SHD-2014,SHD-Lee-2013,GMS-NM-SHD} and the references
therein. On the other hand, Zaprawa \cite{Zaprawa,Zaprawa-AAA} extended the
study of Fekete-Szeg\"{o} problem to certain subclasses of bi-univalent
function class $\sigma.$ Following Zaprawa \cite{Zaprawa,Zaprawa-AAA}, the
Fekete-Szeg\"{o} problem for functions belonging to various other subclasses
of bi-univalent functions were considered in \cite%
{Jay-NM-JY,HO-NM-VKB-Fekete}. Very recently, the upper bounds of $%
H_2(2)$ for the classes $S^*_{\sigma}(\beta)$ and $K_{\sigma}(\beta)$ were
discussed by Deniz et al. \cite{Deniz-SHD}. Recently, Lee et al.%
\cite{SHD-Lee-2013} introduced the following class:
\begin{equation*}
\mathcal{G}^{\lambda}(\varphi):=\left\{ f\in \mathcal{S}:(1-\lambda)f'(z) +\lambda \left(1+\frac{%
zf^{\prime \prime }(z)}{f^{\prime }(z)}\right) \prec \varphi (z);\,\,z\in \mathbb{U}\right\}
\end{equation*}
and obtained the bound for the second Hankel determinant of functions in $\mathcal{G}^{\lambda}(\varphi).$
It is interesting to note that
\begin{equation*}
\mathcal{G}^{\lambda}:= \mathcal{G}^{\lambda} \left(\frac{1+z}{1-z}\right) =
\left \{ f : f \in \mathcal{S} \, \mathrm{and}\, \Re
\left ((1-\lambda)f'(z) +\lambda \left(1+\frac{%
zf^{\prime \prime }(z)}{f^{\prime }(z)}\right)\right ) > 0 ; \,\, z \in
\mathbb{U}\right \}.
\end{equation*}
The class $\mathcal{G}^{\lambda}$ introduced by Al-Amiri and Reade \cite{Al-Amiri-Reade}.
The univalence of the functions in the class $\mathcal{G}^{\lambda}$ was investigated
by Singh et al. \cite{SS-SG-SS-2007,VS-SS-SG-2005}.

Motivated by the recent publications (especially \cite{Azizi,Deniz-SHD,Jay-NM-JY,Peng,Zaprawa,Zaprawa-AAA}), we define the following
subclass of $\sigma .$

\begin{definition}
For $0 \leq \lambda \leq 1$ and $0\leq \beta <1,$ a function $f\in \sigma$
given by (\ref{Int-e1}) is said to be in the class $\mathcal{G}%
^{\lambda}_{\sigma}(\varphi)$ if the following conditions are satisfied:
\[
(1-\lambda)f^{\prime}(z)+\lambda\left(1+\frac{zf^{\prime \prime }(z)}{f^{\prime }(z)}\right)\prec \varphi (z),\qquad
0\leqq \lambda \leqq 1,\,z\in \mathbb{U}
\]
and for $g=f^{-1}$ given by (\ref{Int-f-inver})
\[
(1-\lambda)g^{\prime}(w)+\lambda\left(1+\frac{wg^{\prime \prime }(w)}{g^{\prime }(w)}\right)\prec \varphi (w),\qquad
0\leqq \lambda \leqq 1,\,w\in \mathbb{U}.
\]
\end{definition}
From among the many choices of $\varphi$ and $\lambda$ which would provide the
following known subclasses:
\begin{enumerate}
\item $\mathcal{G}^{0}_{\sigma}(\varphi):%
=\mathcal{H}_{\sigma}(\varphi)$ \cite{Ali-Ravi-Ma-Mina-class},

\item $\mathcal{G}^{1}_{\sigma}(\varphi):%
=\mathcal{K}_{\sigma}(\varphi)$ \cite{Ali-Ravi-Ma-Mina-class},

\item $\mathcal{G}^{\lambda}_{\sigma}(\frac{%
1+(1-2\beta)z}{1-z}):=\mathcal{G}^{\lambda}_{\sigma}(\beta)$ \,\, $(0\leq \beta < 1)$ \,\, \cite{Azizi}.

\item $\mathcal{G}^{0}_{\sigma}(\frac{%
1+(1-2\beta)z}{1-z}):=\mathcal{H}^{\beta}_{\sigma}$ \,\, $(0\leq \beta < 1)$ \,\, \cite{HMS-AKM-PG}

\item $\mathcal{G}^{1}_{\sigma}(\frac{%
1+(1-2\beta)z}{1-z}):=\mathcal{K}_{\sigma}(\beta)$ \,\, $(0\leq \beta < 1)$ \,\, \cite{Brannan-1986}.
\end{enumerate}

In this paper we shall obtain the Fekete-Szeg\"{o} inequalities for $\mathcal{G}%
^{\lambda}_{\sigma}(\varphi)$ as well as its special classes. Further, the second
Hankel determinant obtained for the class $\mathcal{G}^{\lambda}_{\sigma}(\beta).$

\section{Initial Coefficient Bounds}
\begin{theorem}\label{Th-Al-Amiri}
If $f$ given by (\ref{Int-e1}) is in the class $\mathcal{G}^{\lambda}_{\sigma}(\varphi),$ then
\begin{equation}\label{th-|a2|}
|a_2| \leq \frac{B_1\sqrt{B_1}}{\sqrt{4B_1+|(3-\lambda)B_1^2-4B_2|}}
\end{equation}
and
\begin{equation}\label{th-|a3|}
\left\vert a_3\right\vert \leq \left\{
\begin{array}{ll}
\left(1-\frac{4}{3(1+\lambda)B_1}\right)\frac{B_1^3}{4B_1+|(3-\lambda)B_1^2-4B_2|}%
    +\frac{B_1}{3(1+\lambda)},\qquad & \text{if \ \ } B_1 \geq \frac{4}{3(1+\lambda)}{;} \\
&  \\
\frac{B_1}{3(1+\lambda)}, & \text{if \ \ } B_{1} < \frac{4}{3(1+\lambda)}{.}%
\end{array}%
\right.
\end{equation}
\end{theorem}
\begin{proof}
Suppose that $u(z)$ and $v(z)$ are analytic in the unit disk $\mathbb{U}$ with
$u(0)=v(0)=0,$ $|u(z)|<1,$ $|v(z)|<1$ and
\begin{equation}\label{u--v}
u(z)=b_1z+\sum\limits_{n=2}^{\infty}b_nz^n,\,\, v(z)=c_1z+\sum\limits_{n=2}^{\infty}c_nz^n,\qquad  |z|<1.
\end{equation}
It is well known that
\begin{equation}\label{intial-coefs}
|b_1|\leq 1, \,\, |b_2|\leq 1-|b_1|^2,\,\, |c_1|\leq 1, \,\, |c_2|\leq 1-|c_1|^2.
\end{equation}
By a simple calculation, we have
\begin{equation}\label{varphi--u}
\varphi(u(z))=1+B_1b_1z+(B_1b_2+B_2b_1^2)z^2+ \dots ,\,\,\,\, |z|<1
\end{equation}
and
\begin{equation}\label{varphi--v}
\varphi(v(w))=1+B_1c_1w+(B_1c_2+B_2c_1^2)w^2+ \dots ,\,\,\,\, |w|<1.
\end{equation}
Let $f \in \mathcal{G}^{\lambda}_{\sigma}(\varphi).$  Then there are analytic functions
$u,v : \mathbb{U} \rightarrow \mathbb{U}$ given by (\ref{u--v}) such that
\begin{equation}\label{compare--varphiu}
(1-\lambda)f^{\prime}(z)+\lambda\left(1+\frac{zf^{\prime \prime }(z)}{f^{\prime }(z)}\right)=\varphi (u(z))
\end{equation}
\noindent and
\begin{equation}\label{compare--varphiv}
(1-\lambda)g^{\prime}(w)+\lambda\left(1+\frac{wg^{\prime \prime }(w)}{g^{\prime }(w)}\right)=\varphi (v(w)).
\end{equation}
\noindent It follows from (\ref{varphi--u}), (\ref{varphi--v}), (\ref{compare--varphiu}) and (\ref{compare--varphiv}) that
\begin{align}
2a_{2} & = B_{1}b_{1} \label{fa2} \\
3(1+\lambda)a_{3}-4\lambda a_{2}^{2} & = B_{1}b_{2} +B_{2}b_{1}^{2}
\label{fa2a3} \\
-2a_{2} & = B_{1}c_{1} \label{ga2} \\
2(\lambda+3)a_{2}^{2}-3(1+\lambda)a_{3}& = B_{1}c_{2} + B_{2}c_{1}^{2}. \label{ga2a3}
\end{align}
From (\ref{fa2}) and (\ref{ga2}), we get
\begin{equation}\label{b1=c1}
b_1 = -c_1.
\end{equation}
By adding (\ref{fa2a3}) to (\ref{ga2a3}), further, using (\ref{fa2}) and (\ref{b1=c1}), we have
\begin{equation}\label{c2+b2}
(2(3-\lambda)B_{1}^2-8B_{2})a_{2}^2 = B_{1}^3(b_2+c_2).
\end{equation}
In view of (\ref{b1=c1}) and (\ref{c2+b2}), together with (\ref{intial-coefs}), we get
\begin{equation}\label{absolute-c2+b2}
|(2(3-\lambda)B_{1}^2-8B_{2})a_{2}^2| \leq 2 B_{1}^3(1-|b_1|^2).
\end{equation}
Substituting (\ref{fa2}) in (\ref{absolute-c2+b2}) we obtain
\begin{equation}\label{|a2|}
|a_2| \leq \frac{B_1\sqrt{B_1}}{\sqrt{4B_1+|(3-\lambda)B_1^2-4B_2|}}.
\end{equation}
By subtracting (\ref{ga2a3}) from (\ref{fa2a3}) and in view of (\ref{b1=c1}), we get
\begin{equation}\label{a3}
6(1+\lambda)a_{3} = 6(1+\lambda)a_{2}^2 + B_{1}(b_{2}-c_{2}).
\end{equation}
From (\ref{intial-coefs}), (\ref{fa2}), (\ref{b1=c1}) and (\ref{a3}), it follows that
\begin{eqnarray}\label{|a3|}
|a_{3}| & \leq & |a_{2}|^2 + \frac{B_{1}}{6(1+\lambda)}(|b_{2}|+|c_{2}|)\notag\\
        & \leq & |a_{2}|^2 + \frac{B_{1}}{3(1+\lambda)}(1-|b_{1}|^2)\notag\\
        & = & \left(1-\frac{4}{3(1+\lambda)B_1}\right)|a_{2}|^2 + \frac{B_{1}}{3(1+\lambda)}.
\end{eqnarray}
Substituting (\ref{|a2|}) in (\ref{|a3|}) we obtain the desired inequality (\ref{th-|a3|}).
\end{proof}
\begin{remark}
For $\lambda = 0,$ the results obtained in the Theorem \ref{Th-Al-Amiri} are coincide with results in \cite[Theorem 2.1, p.230]{Peng}.
\end{remark}

\begin{corollary}\label{Peng-Th-Cor1}
Let $f \in \mathcal{K}_{\sigma}(\varphi).$ Then
\begin{equation}\label{th-cor1-|a2|}
|a_2| \leq \frac{B_1\sqrt{B_1}}{\sqrt{4B_1+|2B_1^2-4B_2|}}
\end{equation}
and
\begin{equation}\label{th-cor1-|a3|}
\left\vert a_3\right\vert \leq \left\{
\begin{array}{ll}
\left(1-\frac{2}{3B_1}\right)\frac{B_1^3}{4B_1+|2B_1^2-4B_2|}%
    +\frac{B_1}{6}& {;}  B_1 \geq \frac{2}{3}{;} \\
&  \\
\frac{B_1}{3(1+\lambda)} & {;} B_{1} < \frac{2}{3}{.}%
\end{array}%
\right.
\end{equation}
\end{corollary}
\section{Fekete-Szeg\"{o} inequalities}
In order to derive our result, we shall need the following lemma.

\begin{lemma}
\textrm{(see \cite{Duren} or \cite{Jahangiri})}\label{lem-pom} Let $%
p(z)=1+p_{1}z+p_{2}z^{2}+\cdots \in \mathcal{P},$ where $\mathcal{P}$ is the
family of all functions $p,$ analytic in $\mathbb{U},$ for which $\Re
\{p(z)\}>0,$ $z \in \mathbb{U}$. Then
\begin{equation*}
|p_{n}|\leqq 2;\qquad n=1,2,3,...,
\end{equation*}%
and
\begin{equation*}
\left\vert p_{2}-\frac{1}{2}p_{1}^{2}\right\vert \leq 2-\frac{1}{2}%
|p_{1}|^{2}.
\end{equation*}
\end{lemma}

\begin{theorem}
\label{th1-YBIFS} Let $f$ of the form (\ref{Int-e1}) be in $\mathcal{G}%
^{\lambda}_{\sigma}(\varphi).$ Then

\begin{equation}  \label{th1-YBIFS-a2}
\left\vert a_{2}\right\vert \leq \left\{
\begin{array}{ll}
\sqrt{\frac{B_{1}}{3-\lambda}},\qquad & \text{if \ \ }%
|B_{2}|\leq B_{1}{;} \\
&  \\
\sqrt{\frac{|B_{2}|}{3-\lambda}}, & \text{if \ \ }|B_{2}|\geq
B_{1}%
\end{array}%
\right.
\end{equation}

\noindent and

\begin{equation}  \label{th1-YBIFS-a3-a2}
\left\vert a_{3}-\frac{4\lambda}{3+3\lambda}a_{2}^{2}\right\vert
\leq \left\{
\begin{array}{ll}
\frac{B_{1}}{3+3\lambda},\qquad & \text{if \ \ }|B_{2}|\leq B_{1}{;} \\
&  \\
\frac{|B_{2}|}{3+3\lambda}, & \text{if \ \ }|B_{2}|\geq B_{1}{.}%
\end{array}%
\right.
\end{equation}
\end{theorem}

\begin{proof}
Since $f\in \mathcal{G}_{\sigma}^{\lambda}(\varphi),$ there exist two
analytic functions $r,s:\mathbb{U}\rightarrow \mathbb{U},$ with $r(0)=0=s(0),
$ such that
\begin{equation}
(1-\lambda)f^{\prime}(z)+\lambda\left(1+\frac{zf^{\prime \prime }(z)}{f^{\prime }(z)}\right)=\varphi (r(z))
\label{th-YBIFS-p-e3}
\end{equation}%
\noindent and
\begin{equation}
(1-\lambda)g^{\prime}(w)+\lambda\left(1+\frac{wg^{\prime \prime }(w)}{g^{\prime }(w)}\right)=\varphi (s(w)).
\label{th-YBIFS-p-e4}
\end{equation}%
\noindent Define the functions $p$ and $q$ by
\begin{equation*}
p(z)=\frac{1+r(z)}{1-r(z)}=1+p_{1}z+p_{2}z^{2}+p_{3}z^{3}+\dots
\end{equation*}%
\noindent and
\begin{equation*}
q(w)=\frac{1+s(w)}{1-s(w)}=1+q_{1}w+q_{2}w^{2}+q_{3}w^{3}+\dots
\end{equation*}%
\noindent or equivalently,
\begin{equation}
r(z)=\frac{p(z)-1}{p(z)+1}=\frac{1}{2}\left( p_{1}z+\left( p_{2}-\frac{%
p_{1}^{2}}{2}\right) z^{2}+\left( p_{3}+\frac{p_{1}}{2}\left( \frac{p_{1}^{2}%
}{2}-p_{2}\right) -\frac{p_{1}p_{2}}{2}\right) z^{3}+\dots \right)
\label{th-YBIFS-p-e7}
\end{equation}%
\noindent and
\begin{equation}
s(w)=\frac{q(w)-1}{q(w)+1}=\frac{1}{2}\left( q_{1}w+\left( q_{2}-\frac{%
q_{1}^{2}}{2}\right) w^{2}+\left( q_{3}+\frac{q_{1}}{2}\left( \frac{q_{1}^{2}%
}{2}-q_{2}\right) -\frac{q_{1}q_{2}}{2}\right) w^{3}+\dots \right) .
\label{th-YBIFS-p-e8}
\end{equation}%
\noindent Using (\ref{th-YBIFS-p-e7}) and (\ref{th-YBIFS-p-e8}) in (\ref%
{th-YBIFS-p-e3}) and (\ref{th-YBIFS-p-e4}), we have
\begin{equation}
(1-\lambda)f^{\prime}(z)+\lambda\left(1+\frac{zf^{\prime \prime }(z)}{f^{\prime }(z)}\right)=\varphi \left( \frac{%
p(z)-1}{p(z)+1}\right)   \label{th-YBIFS-p-e3-9}
\end{equation}%
\noindent and
\begin{equation}
(1-\lambda)g^{\prime}(w)+\lambda\left(1+\frac{wg^{\prime \prime }(w)}{g^{\prime }(w)}\right)=\varphi \left( \frac{%
q(w)-1}{q(w)+1}\right) .  \label{th-YBIFS-p-e3-10}
\end{equation}

\noindent Again using (\ref{th-YBIFS-p-e7}) and (\ref{th-YBIFS-p-e8}) along
with (\ref{varphi-expression}), it is evident that

\begin{equation}
\varphi \left( \frac{p(z)-1}{p(z)+1}\right) =1+\frac{1}{2}B_{1}p_{1}z+\left(
\frac{1}{2}B_{1}\left( p_{2}-\frac{1}{2}p_{1}^{2}\right) +\frac{1}{4}%
B_{2}p_{1}^{2}\right) z^{2}+\dots   \label{th-YBIFS-p-e11}
\end{equation}%
\noindent and
\begin{equation}
\varphi \left( \frac{q(w)-1}{q(w)+1}\right) =1+\frac{1}{2}B_{1}q_{1}w+\left(
\frac{1}{2}B_{1}\left( q_{2}-\frac{1}{2}q_{1}^{2}\right) +\frac{1}{4}%
B_{2}q_{1}^{2}\right) w^{2}+\dots .  \label{th-YBIFS-p-e12}
\end{equation}%
\noindent It follows from (\ref{th-YBIFS-p-e3-9}), (\ref{th-YBIFS-p-e3-10}),
(\ref{th-YBIFS-p-e11}) and (\ref{th-YBIFS-p-e12}) that
\begin{eqnarray}
2a_{2}&=&\frac{1}{2}B_{1}p_{1} \notag \\
3(1+\lambda)a_{3}-4\lambda a_{2}^{2}&=&\frac{1}{2}B_{1}\left(
p_{2}-\frac{1}{2}p_{1}^{2}\right) +\frac{1}{4}B_{2}p_{1}^{2}
\label{th-YBIFS-p-e14}\\
-2a_{2}&=&\frac{1}{2}B_{1}q_{1}
\notag \\
2(\lambda+3)a_{2}^{2}-3(1+\lambda)a_{3}&=&\frac{1}{2}B_{1}\left(
q_{2}-\frac{1}{2}q_{1}^{2}\right) +\frac{1}{4}B_{2}q_{1}^{2}.
\label{th-YBIFS-p-e16}
\end{eqnarray}%
\noindent 

\noindent Dividing (\ref{th-YBIFS-p-e14}) by $3+3\lambda$ and taking the
absolute values we obtain
\begin{equation*}
\left\vert a_{3}-\frac{4\lambda}{3+3\lambda}a_{2}^{2}\right\vert
\leq \frac{B_{1}}{6+6\lambda}\left\vert p_{2}-\frac{1}{2}%
p_{1}^{2}\right\vert +\frac{|B_{2}|}{12+12\lambda}|p_{1}|^{2}.
\end{equation*}

\noindent Now applying Lemma \ref{lem-pom}, we have

\begin{equation*}
\left\vert a_{3}-\frac{4\lambda}{3+3\lambda}a_{2}^{2}\right\vert
\leq \frac{B_{1}}{3+3\lambda}+\frac{|B_{2}|-B_{1}}{12+12\lambda}|p_{1}|^{2}.
\end{equation*}

\noindent Therefore

\begin{equation*}
\left\vert a_{3}-\frac{4\lambda}{3+3\lambda}a_{2}^{2}\right\vert
\leq \left\{
\begin{array}{ll}
\frac{B_{1}}{3+3\lambda},\qquad  & \text{if \ \ }|B_{2}|\leq B_{1}{;} \\
&  \\
\frac{|B_{2}|}{3+3\lambda}, & \text{if \ \ }|B_{2}|\geq B_{1}{.}%
\end{array}%
\right.
\end{equation*}%
\noindent Adding (\ref{th-YBIFS-p-e14}) and (\ref{th-YBIFS-p-e16}), we have
\begin{equation}
(6-2\lambda)a_{2}^{2}=\frac{B_{1}}{2}(p_{2}+q_{2})-\frac{%
(B_{1}-B_{2})}{4}(p_{1}^{2}+q_{1}^{2}).  \label{JJM-a-2-square}
\end{equation}

\noindent Dividing (\ref{JJM-a-2-square}) by $6-2\lambda$ and taking
the absolute values we obtain
\begin{equation*}
|a_{2}|^{2}\leq \frac{1}{6-2\lambda}\left[ \frac{B_{1}}{2}%
\left\vert p_{2}-\frac{1}{2}p_{1}^{2}\right\vert +\frac{|B_{2}|}{4}%
|p_{1}|^{2}+\frac{B_{1}}{2}\left\vert q_{2}-\frac{1}{2}q_{1}^{2}\right\vert +%
\frac{|B_{2}|}{4}|q_{1}|^{2}\right] .
\end{equation*}

\noindent Once again, apply Lemma \ref{lem-pom} to obtain

\begin{equation*}
|a_{2}|^{2}\leq \frac{1}{6-2\lambda}\left[ \frac{B_{1}}{2}%
\left( 2-\frac{1}{2}|p_{1}|^{2}\right) +\frac{|B_{2}|}{4}|p_{1}|^{2}+\frac{%
B_{1}}{2}\left( 2-\frac{1}{2}|q_{1}|^{2}\right) +\frac{|B_{2}|}{4}|q_{1}|^{2}%
\right] .
\end{equation*}

\noindent Upon simplification we obtain

\begin{equation*}
|a_{2}|^{2}\leq \frac{1}{6-2\lambda}\left[ 2B_{1}+\frac{%
|B_{2}|-B_{1}}{2}\left( |p_{1}|^{2}+|q_{1}|^{2}\right) \right] .
\end{equation*}
\noindent Therefore

\begin{equation*}
\left\vert a_{2}\right\vert \leq \left\{
\begin{array}{ll}
\sqrt{\frac{B_{1}}{3-\lambda}},\qquad & \text{if \ \ }%
|B_{2}|\leq B_{1}{;} \\
&  \\
\sqrt{\frac{|B_{2}|}{3-\lambda}}, & \text{if \ \ }|B_{2}|\geq
B_{1}%
\end{array}%
\right.
\end{equation*}
which completes the proof.
\end{proof}
\begin{remark}
Taking
\begin{equation}
\varphi (z)=\left( \frac{1+z}{1-z}\right) ^{\beta }=1+2\beta z+2\beta
^{2}z^{2}+\dots, \qquad 0<\beta \leq 1  \label{varphi-beta}
\end{equation}%
the inequalities (\ref{th1-YBIFS-a2}) and (\ref{th1-YBIFS-a3-a2}) become
\begin{equation}
|a_{2}|\leq \sqrt{\frac{2\beta}{3-\lambda}}\qquad \mathrm{and%
}\qquad \left\vert a_{3}-\frac{4\lambda}{3+3\lambda}%
a_{2}^{2}\right\vert \leq \frac{2\beta}{3+3\lambda}.  \label{th1-remark1}
\end{equation}%
\noindent For
\begin{equation}  \label{varphi-alpha}
\varphi(z) = \frac{1+(1-2\beta)z}{1-z} = 1 + 2(1-\beta) z + 2(1-\beta)
z^2 + \dots,  \qquad 0 \leq \beta < 1
\end{equation}
the inequalities (\ref{th1-YBIFS-a2}) and (\ref{th1-YBIFS-a3-a2}) become
\begin{equation}  \label{th1-remark2}
|a_2| \leq \sqrt{\frac{2(1-\beta)}{3-\lambda}} \qquad \mathrm{%
and} \qquad \left\vert a_{3}-\frac{4\lambda}{3-\lambda}%
a_{2}^{2}\right\vert \leq \frac{2(1-\beta)}{3+3\lambda} .
\end{equation}
\end{remark}

\section{Bounds for the second Hankel determinant of $\mathcal{G}^{\lambda}_{\sigma}(\beta)$}

Next we state the following lemmas  to establish the desired
bounds in our study.

\begin{lemma}
\cite{Pom}\label{L-Repart-fun-p} If the function $p\in \mathcal{P}$ is given
by the series
\begin{equation}  \label{Repart-fun-p}
p(z)= 1+ p_{1} z + p_{2} z^{2} + p_{3} z^{3} + \cdots,
\end{equation}
then the following sharp estimate holds:
\begin{equation}  \label{cnbound}
|p_{n}|\leq 2, \qquad n= 1,2, \cdots.
\end{equation}
\end{lemma}

\begin{lemma}
\label{L-C2-C3} \cite{Grender} If the function $p\in \mathcal{P}$ is given
by the series (\ref{Repart-fun-p}), then
\begin{eqnarray*}
2c_2 &=& c_1^2 + x (4-c_1^2)\,\,  \label{L2-c2} \\
4c_3 &=& c_1^3+2c_1(4-c_1^2)x-c_1(4-c_1^2)x^2+2(4-c_1^2)(1-|x|^2)z\,\,
\label{L2-c3}
\end{eqnarray*}
for some $x,$ $z$ with $|x| \leq 1$ and $|z| \leq 1.$
\end{lemma}

The following theorem provides a bound for the second Hankel determinant
of the functions in the class $\mathcal{G}^{\lambda}_{\sigma}(\beta).$


\begin{theorem}
\label{th-SHD-class} Let $f$ of the form (\ref{Int-e1}) be in $\mathcal{G}^{\lambda}_{\sigma}(\beta).$ Then

\[
|a_2a_4-a_3^2| \leq \left \{
\begin{array}{ll}
\frac{(1-\beta)^2}{2(1+2\lambda)} \left[(2-\lambda)(1-\beta)^2+1
\right]~; &  \\
\qquad \qquad\qquad\beta \in \left[0, 1- \frac{(1+2\lambda)+\sqrt{
(1+2\lambda)^2+18(1+\lambda)^2(2-\lambda)}}{6(1+\lambda)(2-\lambda)}
\right] &   \\
\tfrac{(1-\beta)^2}{72(1+2\lambda)}\left(\tfrac{\begin{array}{ll}
        36[8(1+2\lambda)(2-\lambda)-(1+2\lambda)^2](1-\beta)^2
        \\[4mm] \,\,- 324(1+\lambda)(1+2\lambda)(1-\beta)
             +288(1+2\lambda)-729(1+\lambda)^2\end{array}}
        {\begin{array}{ll}
        9(1+\lambda)^2(2-\lambda)(1-\beta)^2
            - 6(1+\lambda)(1+2\lambda)(1-\beta)
        \\[4mm] \,\, +8(1+2\lambda)-18(1+\lambda)^2
        \end{array}}\right)~;
        & \\
        \qquad \qquad\qquad \beta \in \left(1- \frac{(1+2\lambda)+\sqrt{
        (1+2\lambda)^2+18(1+\lambda)^2(2-\lambda)}}
        {6(1+\lambda)(2-\lambda)},1\right).
\end{array}
\right.
\]
\end{theorem}

\begin{proof}
Let $f\in \mathcal{G}^{\lambda}_{\sigma}(\beta).$ Then
\begin{equation}  \label{SHD-th-p-e1}
(1-\lambda)f^{\prime}(z)+\lambda\left(1+\frac{zf^{\prime \prime
}(z)}{f^{\prime }(z)}\right )=\beta+(1-\beta)p(z)
\end{equation}
and
\begin{equation}  \label{SHD-th-p-e2}
(1-\lambda)g^{\prime}(w)+\lambda\left(1+\frac{wg^{\prime \prime
}(w)}{g^{\prime }(w)}\right )=\beta+(1-\beta)q(w),
\end{equation}
where $p, q \in \mathcal{P}$ and defined by
\begin{equation}  \label{SHD-th-p-e3}
p(z)=1+c_1z+c_2z^2+c_3z^3+\dots
\end{equation}
and
\begin{equation}  \label{SHD-th-p-e4}
q(z)=1+d_1w+d_2w^2+d_3w^3+\dots {.}
\end{equation}
It follows from (\ref{SHD-th-p-e1}), (\ref{SHD-th-p-e2}), (\ref{SHD-th-p-e3}%
) and (\ref{SHD-th-p-e4}) that
\begin{eqnarray}
2a_2 &=& (1-\beta)c_1  \label{SHD-th-p-e5} \\
3(1+\lambda)a_3-4\lambda a_2^2 &=& (1-\beta)c_2  \label{SHD-th-p-e6} \\
4(1+2\lambda)a_4-18\lambda a_2a_3+ 8\lambda a_2^3 &=& (1-\beta)c_3
\label{SHD-th-p-e7}
\end{eqnarray}
and
\begin{eqnarray}
-2a_2 &=& (1-\beta)d_1  \label{SHD-th-p-e8} \\
2(3+\lambda)a_2^2-3(1+\lambda)a_3 &=& (1-\beta)d_2  \label{SHD-th-p-e9} \\
2(10+11\lambda)a_2a_3-4(5+3\lambda)a_2^3-4(1+2\lambda)a_4 &=& (1-\beta)d_3.
\label{SHD-th-p-e10}
\end{eqnarray}
From (\ref{SHD-th-p-e5}) and (\ref{SHD-th-p-e8}), we find that
\begin{equation}  \label{SHD-th-p-e11}
c_1 =-d_1
\end{equation}
and
\begin{equation}  \label{SHD-a2}
a_2=\frac{1-\beta}{2}c_1.
\end{equation}
Now, from (\ref{SHD-th-p-e6}), (\ref{SHD-th-p-e9}) and (\ref{SHD-a2}), we
have
\begin{equation}  \label{SHD-a3}
a_3=\frac{(1-\beta)^2}{4}c_1^2+ \frac{1-\beta}{6(1+\lambda)}%
(c_2-d_2).
\end{equation}
Also, from (\ref{SHD-th-p-e7}) and (\ref{SHD-th-p-e10}), we find that
\begin{equation}  \label{SHD-a4}
a_4=\frac{5\lambda(1-\beta)^3}{16(1+2\lambda)}c_1^3 +\frac{%
5(1-\beta)^2}{24(1+\lambda)}c_1(c_2-d_2) +\frac{1-\beta}{%
8(1+2\lambda)}(c_3-d_3).
\end{equation}
Then, we can establish that
\begin{eqnarray}  \label{H2of2}
|a_2a_4-a_3^2| &=& \left |\frac{(\lambda-2)(1-\beta)^4}{%
32(1+2\lambda)}c_1^4 +\frac{(1-\beta)^3}{48(1+\lambda)}%
c_1^2(c_2-d_2)\right.  \notag \\
&& \left. \quad +\frac{(1-\beta)^2}{16(1+2\lambda)}c_1(c_3-d_3) -%
\frac{(1-\beta)^2}{36(1+\lambda)^2}(c_2-d_2)^2 \right |.
\end{eqnarray}
According to Lemma \ref{L-C2-C3} and (\ref{SHD-th-p-e11}), we write

\begin{eqnarray}
c_2-d_2 &=& \frac{(4-c_1^2)}{2}(x-y) \label{c2=d2} \\
c_3-d_3 &=& \frac{c_1^3}{2}+\frac{c_1(4-c_1^2)(x+y)}{2}-\frac{%
c_1(4-c_1^2)(x^2+y^2)}{4} \nonumber \\ && \qquad +\frac{(4-c_1^2)[(1-|x|^2)z-(1-|y|^2)w]}{2}%
\label{c3-d3}
\end{eqnarray}
for some $x,y,z$ and $w$ with $|x|\leq 1,$ $|y|\leq 1,$ $|z|\leq 1$ and $%
|w|\leq 1.$ Using (\ref{c2=d2}) and (\ref{c3-d3}) in (\ref{H2of2}), we have

\begin{eqnarray}  \label{H2-2}
|a_2a_4-a_3^2| &=& \left | \frac{(\lambda-2)(1-\beta)^4c_1^4}{32(1+2\lambda)}
+\frac{(1-\beta)^3c_1^2(4-c_1^2)(x-y)}{96(1+\lambda)} +\frac{%
(1-\beta)^2c_1}{16(1+2\lambda)} \right.  \notag \\
&& \left. \quad \times \left[\frac{c_1^3}{2}+\frac{c_1(4-c_1^2)(x+y)}{2}-%
\frac{c_1(4-c_1^2)(x^2+y^2)}{4}\right. \right.  \notag \\
&& \left. \left. \qquad +\frac{(4-c_1^2)[(1-|x|^2)z-(1-|y|^2)w]}{2}\right]-%
\frac{(1-\beta)^2(4-c_1^2)^2}{144(1+\lambda)^2}(x-y)^2 \right |  \notag \\
&\leq& \frac{(2-\lambda)(1-\beta)^4}{32(1+2\lambda)}c_1^4 +\frac{%
(1-\beta)^2c_1^4}{32(1+2\lambda)} + \frac{(1-\beta)^2c_1(4-c_1^2)}{%
16(1+2\lambda)}  \notag \\
&& +\left[\frac{(1-\beta)^3c_1^2(4-c_1^2)}{96(1+\lambda)} +\frac{%
(1-\beta)^2c_1^2(4-c_1^2)}{32(1+2\lambda)}\right](|x|+|y|)  \notag
\\
&& +\left[\frac{(1-\beta)^2c_1^2(4-c_1^2)}{64(1+2\lambda)} -\frac{%
(1-\beta)^2c_1(4-c_1^2)}{32(1+2\lambda)}\right](|x|^2+|y|^2)  \notag
\\
&& \quad +\frac{(1-\beta)^2(4-c_1^2)^2}{144(1+\lambda)^2}(|x|+|y|)^2.  \notag
\end{eqnarray}

Since $p\in \mathcal{P},$ so $|c_{1}|\leq 2.$ Letting $c_{1}=c,$ we may
assume without restriction that $c\in \lbrack 0,2].$ Thus, for $\gamma
_{1}=|x|\leq 1$ and $\gamma _{2}=|y|\leq 1,$ we obtain
\begin{equation*}
|a_{2}a_{4}-a_{3}^{2}|\leq T_{1}+T_{2}(\gamma _{1}+\gamma _{2})+T_{3}(\gamma
_{1}^{2}+\gamma _{2}^{2})+T_{4}(\gamma _{1}+\gamma _{2})^{2}=F(\gamma
_{1},\gamma _{2}),
\end{equation*}%
\begin{eqnarray*}
T_{1} &=&T_{1}(c)=\frac{(2-\lambda)(1-\beta)^{4}}{32(1+2\lambda)}c^{4}+%
\frac{(1-\beta)^{2}c^{4}}{32(1+2\lambda)}+\frac{(1-\beta)^{2}c(4-c^{2})}%
{16(1+2\lambda)}\geq 0 \\ %
T_{2} &=&T_{2}(c)=\frac{(1-\beta)^{3}c^{2}(4-c^{2})}{96(1+\lambda)}%
+\frac{(1-\beta)^{2}c^{2}(4-c^{2})}{32(1+2\lambda)}\geq 0 \\ %
T_{3} &=&T_{3}(c)=\frac{(1-\beta)^{2}c^{2}(4-c^{2})}{64(1+2\lambda)}-%
\frac{(1-\beta)^{2}c(4-c^{2})}{32(1+2\lambda)}%
\leq 0 \\
T_{4} &=&T_{4}(c)=\frac{(1-\beta)^{2}(4-c^{2})^{2}}{144(1+\lambda)^{2}}%
\geq 0.
\end{eqnarray*}%

Now we need to maximize $F(\gamma _{1},\gamma _{2})$ in the closed square $%
\mathbb{S}:=\{(\gamma _{1},\gamma _{2}):0\leq \gamma _{1}\leq 1,0\leq \gamma
_{2}\leq 1\}$ for $c\in \lbrack 0,2]$. We must investigate the maximum of $%
F(\gamma _{1},\gamma _{2})$ according to $c\in (0,2),$ $c=0$ and $c=2$
taking into account the sign of $F_{\gamma _{1}\gamma _{1}}F_{\gamma
_{2}\gamma _{2}}-(F_{\gamma _{1}\gamma _{2}})^{2}.$

Firstly, let $c\in (0,2).$ Since $T_{3}<0$ and $T_{3}+2T_{4}>0$ for $c\in
(0,2),$ we conclude that
\begin{equation*}
F_{\gamma _{1}\gamma _{1}}F_{\gamma _{2}\gamma _{2}}-(F_{\gamma _{1}\gamma
_{2}})^{2}<0.
\end{equation*}

Thus, the function $F$ cannot have a local maximum in the interior of the
square $\mathbb{S}.$ Now, we investigate the maximum of $F$ on the boundary
of the square $\mathbb{S}.$

For $\gamma_1=0$ and $0\leq \gamma_2 \leq 1$ (similarly $\gamma_2=0$ and $%
0\leq \gamma_1\leq1$) we obtain
\begin{equation*}
F(0,\gamma_2)=G(\gamma_2)=T_1+T_2\gamma_2+(T_3+T_4)\gamma_2^2 {.}
\end{equation*}

(i) The case $T_{3}+T_{4}\geq 0:$ In this case for $0<\gamma _{2}<1$ and any
fixed $c$ with $0<c<2,$ it is clear that $G^{\prime }(\gamma
_{2})=2(T_{3}+T_{4})\gamma _{2}+T_{2}>0,$ that is, $G(\gamma _{2})$ is an
increasing function. Hence, for fixed $c\in (0,2),$ the maximum of $G(\gamma
_{2})$ occurs at $\gamma _{2}=1$ and
\begin{equation*}
\max G(\gamma _{2})=G(1)=T_{1}+T_{2}+T_{3}+T_{4}.
\end{equation*}

(ii) The case $T_{3}+T_{4}<0:$ Since $T_{2}+2(T_{3}+T_{4})\geq 0$ for $%
0<\gamma _{2}<1$ and any fixed $c$ with $0<c<2,$ it is clear that $%
T_{2}+2(T_{3}+T_{4})<2(T_{3}+T_{4})\gamma _{2}+T_{2}<T_{2}$ and so $%
G^{\prime }(\gamma _{2})>0.$ Hence for fixed $c\in (0,2),$ the maximum of $%
G(\gamma _{2})$ occurs at $\gamma _{2}=1$ and also for $c=2$ we obtain
\begin{equation}\label{F}
F(\gamma_1, \gamma_2) = \frac{(1-\beta)^2}{2(1+2\lambda)} \left[%
(2-\lambda)(1-\beta)^2+1\right].
\end{equation}

Taking into account the value (\ref{F}) and the cases $i$ and $ii,$ for $%
0\leq \gamma_2 <1$ and any fixed $c$ with $0\leq c \leq 2$ we have
\begin{equation*}
\max G(\gamma_2) = G(1) = T_1+T_2+T_3+T_4.
\end{equation*}

For $\gamma_1=1$ and $0\leq \gamma_2 \leq 1$ (similarly $\gamma_2=1$ and $%
0\leq \gamma_1 \leq 1$), we obtain
\begin{equation*}
F(1,\gamma_2)=H(\gamma_2)=(T_3+T_4)\gamma_2^2 + (T_2+2T_4)\gamma_2 + T_1 +
T_2 + T_3 + T_4.
\end{equation*}

\noindent Similarly, to the above cases of $T_3+T_4,$ we get that
\begin{equation*}
\max H(\gamma_2) = H(1) = T_1+2T_2+2T_3+4T_4.
\end{equation*}

\noindent Since $G(1)\leq H(1)$ for $c\in (0,2),$ $\max F(\gamma _{1},\gamma
_{2})=F(1,1)$ on the boundary of the square $\mathbb{S}.$ Thus the maximum
of $F$ occurs at $\gamma _{1}=1$ and $\gamma _{2}=1$ in the closed square $%
\mathbb{S}.$

\noindent Let $K:(0,2)\rightarrow \mathbb{R}$
\begin{equation}
K(c)=\max F(\gamma _{1},\gamma _{2})=F(1,1)=T_{1}+2T_{2}+2T_{3}+4T_{4}.
\label{K}
\end{equation}

\noindent Substituting the values of $T_1,$ $T_2,$ $T_3$ and $T_4$ in the function $K$
defined by (\ref{K}), yields
\begin{align*}
K(c) &= \frac{(1-\beta)^2}{288(1+\lambda)^2(1+2\lambda)}\left\{%
\left[9(1-\beta)^2(1+\lambda)^2(2-\lambda) \right.\right. \\
& \left.\left.\qquad -6(1-\beta)(1+\lambda)(1+2\lambda)
-18(1+\lambda)^2+8(1+2\lambda)\right] c^4 \right. \\
& \left. \qquad + \left[24(1-\beta)(1+\lambda)(1+2\lambda)
+108(1+\lambda)^2-64(1+2\lambda)\right] c^2 \right. \\
& \qquad \left.+ 128(1+2\lambda)\right\}.
\end{align*}

\noindent Assume that $K(c)$ has a maximum value in an interior of $c\in (0,2),$ by
elementary calculation, we find
\begin{align*}
K^{\prime}(c) &= \frac{(1-\beta)^2}{72(1+\lambda)^2(1+2\lambda)}\left\{%
\left[9(1-\beta)^2(1+\lambda)^2(2-\lambda) \right.\right. \\
& \left.\left.\qquad -6(1-\beta)(1+\lambda)(1+2\lambda)
-18(1+\lambda)^2+8(1+2\lambda)\right] c^3 \right. \\
& \left. \qquad + \left[12(1-\beta)(1+\lambda)(1+2\lambda)
+54(1+\lambda)^2-32(1+2\lambda)\right] c \right\}.
\end{align*}

After some calculations we concluded the following cases:

\begin{case}
Let
\begin{equation*}
[9(1-\beta)^2(1+\lambda)^2(2-\lambda) -6(1-\beta)(1+\lambda)(1+2\lambda)
-18(1+\lambda)^2+8(1+2\lambda)]\geq 0,
\end{equation*}

\noindent that is,

\begin{equation*}
\beta \in \left[ 0,1-\frac{(1+2\lambda)+\sqrt{(1+2\lambda)^{2}+(2-\lambda)[18(1+\lambda)^2
-8(1+2\lambda)]}}{3(1+\lambda)(2-\lambda)}
\right] {.}
\end{equation*}%
Therefore $K^{\prime }(c)>0$ for $c\in (0,2).$ Since $K$ is an increasing
function in the interval $(0,2),$ maximum point of $K$ must be on the
boundary of $c\in [0,2],$ that is, $c=2.$ Thus, we have
\begin{equation*}
\max\limits_{0<c<2}K(c)=K(2)=\frac{(1-\beta )^{2}}{2(1+2\lambda)}\left[%
(2-\lambda)(1-\beta)^{2}+1\right] .
\end{equation*}
\end{case}

\begin{case}
Let
\begin{equation*}
[9(1-\beta)^2(1+\lambda)^2(2-\lambda) -6(1-\beta)(1+\lambda)(1+2\lambda)
-18(1+\lambda)^2+8(1+2\lambda)] < 0,
\end{equation*}

\noindent that is,

\begin{equation*}
\beta \in \left[1-\frac{(1+2\lambda)+\sqrt{(1+2\lambda)^{2}+(2-\lambda)[18(1+\lambda)^2
-8(1+2\lambda)]}}{3(1+\lambda)(2-\lambda)}, 1\right].
\end{equation*}
\noindent Then $K^{\prime }(c)=0$ implies the real critical point $c_{0_1}=0$ or
\begin{equation*}
c_{0_2} = \sqrt{\frac{-12(1+\lambda)(1+2\lambda)(1-\beta)%
-54(1+\lambda)^2+32(1+2\lambda)} {9(1-\beta)^2(1+\lambda)^2(2-\lambda) -6(1-\beta)(1+\lambda)(1+2\lambda)
-18(1+\lambda)^2+8(1+2\lambda)}} .
\end{equation*}
\noindent When
\begin{equation*}
\beta \in \left(1-\tfrac{(1+2\lambda)+\sqrt{(1+2\lambda)^{2}+(2-\lambda)[18(1+\lambda)^2
-8(1+2\lambda)]}}{3(1+\lambda)(2-\lambda)} \right., \left.1-
\tfrac{(1+2\lambda)+\sqrt{(1+2\lambda)^2+18(1+\lambda)^2(2-\lambda)}]}{%
6(1+\lambda)(2-\lambda)}\right].
\end{equation*}

We observe that $c_{0_2}\geq 2,$ that is, $c_{0_2}$ is out of the interval $%
(0,2).$ Therefore, the maximum value of $K(c)$ occurs at $c_{0_1}=0$ or $%
c=c_{0_2}$ which contradicts our assumption of having the maximum value at
the interior point of $c \in [0,2].$ Since $K$ is an increasing function in
the interval $(0,2),$ maximum point of $K$ must be on the boundary of
$c\in[0,2]$ that is $c=2.$ Thus, we have
\[
\max\limits_{0\leq c \leq 2} K(c) = K(2) = \frac{(1-\beta)^2}{2(1+2\lambda)}%
[1+(2-\lambda)(1-\beta)^2].
\]
\noindent When $\beta \in \left(1-
\tfrac{(1+2\lambda)+\sqrt{(1+2\lambda)^2+18(1+\lambda)^2(2-\lambda)}]}{
6(1+\lambda)(2-\lambda)}, 1\right),$ we observe
that $c_{0_2}< 2,$ that is, $c_{0_2}$ is an interior of the interval $[0,2].$
Since $K^{\prime \prime }(c_{0_2})<0,$ the maximum value of $K(c)$ occurs at
$c=c_{0_2}.$ Thus, we have
\begin{eqnarray*}
\max\limits_{0 \leq c \leq 2} K(c) &=& K(c_{0_2}) \\
&=& \frac{(1-\beta)^2}{72(1+2\lambda)}\left(\tfrac{\begin{array}{ll}%
        36[8(1+2\lambda)(2-\lambda)-(1+2\lambda)^2](1-\beta)^2%
        \\[4mm] \,\,- 324(1+\lambda)(1+2\lambda)(1-\beta)%
             +288(1+2\lambda)-729(1+\lambda)^2\end{array}}%
        {\begin{array}{ll}%
        9(1+\lambda)^2(2-\lambda)(1-\beta)^2%
        \\[4mm] \,\,- 6(1+\lambda)(1+2\lambda)(1-\beta)%
             +8(1+2\lambda)-18(1+\lambda)^2%
        \end{array}}\right)~{.} %
\end{eqnarray*}
\end{case}
\noindent This completes the proof.
\end{proof}

\begin{corollary}\label{SHD-Cor-1}
Let $f$ of the form (\ref{Int-e1}) be in $\mathcal{H}_{\sigma}^{\beta}.$ Then
\[
|a_2a_4-a_3^2| \leq
\left\{
  \begin{array}{ll}
    \frac{(1-\beta)^2[1+2(1-\beta)^2]}{2}~{;} &  \beta \in \left[0, \frac{11-\sqrt{37}}{12}\right] \\
    \frac{(1-\beta)^2[60\beta^2-84\beta-25]}{16(9\beta^2-15\beta+1)}~{;} &  \beta \in \left(\frac{11-\sqrt{37}}{12}, 1\right)
    {.}
  \end{array}
\right.
\]
\end{corollary}
\begin{corollary}\label{SHD-Cor-2}
Let $f$ of the form (\ref{Int-e1}) be in $\mathcal{K}_{\sigma}(\beta).$ Then
\[
|a_2a_4-a_3^2| \leq \frac{(1-\beta)^2}{24}\left[\frac{5\beta^2+8\beta-32}{3\beta^2-3\beta-4}\right] {.}
\]
\end{corollary}
\begin{corollary}\label{SHD-Cor-3}
Let $f$ of the form (\ref{Int-e1}) be in $\mathcal{H}_{\sigma}.$ Then
\begin{equation*}
|a_2a_4-a_3^2| \leq \frac{3}{2}~.
\end{equation*}
\end{corollary}

\begin{corollary}\label{SHD-Cor-4}
Let $f$ of the form (\ref{Int-e1}) be in $\mathcal{K}_{\sigma}.$ Then
\begin{equation*}
|a_2a_4-a_3^2| \leq \frac{1}{3}~.
\end{equation*}
\end{corollary}
%


\end{document}